\documentclass[reqno,b5paper]{amsart}
\usepackage{amsmath}
\usepackage{amssymb}
\usepackage{amsthm}
\usepackage{enumerate}
\usepackage[mathscr]{eucal}
\theoremstyle{plain}
\newtheorem{thm}{Theorem}[section]

\newtheorem{lem}[thm]{Lemma}
\newtheorem{note}{Note}[section]

\theoremstyle{definition}
\newtheorem{defn}{Definition}[section]
\newtheorem{rem}{Remark}[section]

\begin{document}
\setcounter {page}{1}
\title{ Strong $I$ AND $I^*$-statistically pre-Cauchy double sequences in Probabilistic Metric Spaces}

\author[P. Malik, A. Ghosh AND M. Maity]{ Prasanta Malik*, Argha Ghosh* and Manojit Maity** \ }
\newcommand{\acr}{\newline\indent}
\maketitle
\address{{*\,} Department of Mathematics, The University of Burdwan, Golapbag, Burdwan-713104,
West Bengal, India.
                Email: pmjupm@yahoo.co.in., buagbu@yahoo.co.in\acr
           {**\,} Boral High School, Kolkata-700154, India. Email: mepsilon@gmail.com\\}

\maketitle
\begin{abstract}In this paper we consider the notion of strong $I$-statistically pre-Cauchy
double sequences in probabilistic metric spaces in line of Das et.
al. [6] and introduce the new concept of strong
$I^*$-statistically pre-Cauchy double sequences in real line as
well as in probabilistic metric spaces. We mainly study inter
relationship among strong $I$-statistical convergence, strong
$I$-statistical pre-Cauchy condition and strong $I^*$-statistical
pre-Cauchy condition for double sequences in probabilistic metric
spaces and examine some basic properties of these notions.
\end{abstract}
\author{}
\maketitle


\textbf{
\begin{center}
1. Introduction
\end{center}
} The notion of probabilistic metric (PM) was introduced by K.
Menger [10] under the name of ``statistical metric spaces" by
considering the distance between two pints $x$ and $y$ as a
distribution function $F_{xy}$ instead of a non-negative real
number and the value of the function $F_{xy}$ at any $t> 0 $ i.e.
$F_{xy} (t) $ is interpreted as the probability that the distance
between the points $x$ and $y$ is $\leq t$. After Menger, the
theory of probabilistic metric was developed by Schwiezer and
Sklar ([16], [17], [18], [19]), Tardiff [23], Thorp [24] and many
others. A through development of probabilistic metric spaces ( can
be seen from the famous book Schwiezer and Sklar [20]. Many
different topologies and it is the main tool of our paper.

The idea of usual notion of convergence of real sequences was
extended to statistical convergence by Fast [8] and Schoenberg
[15] independently. For the last few years a lot of work has been
done on this convergence ( see [2], [8], [9], [14], [22] etc. ).
In [2] the notion of statistically pre-Cauchy sequences of real
numbers was introduced and it was shown that statistically
convergent sequences are always statistically pre-Cauchy and the
converse statement holds under certain conditions. The notion of
statistical convergence was further extended to $I$-convergence
[13] using the ideals of $\mathbb{N}$ and also to $I$-statistical
convergence in [14]. The notion of strong $I$-statistical
convergence for double sequences of real numbers had been
introduced by Belen et. al. in [1]. Recently in [6] Das and Savas
introduced the notion of $I$-statistically pre-Cauchy sequence of
real numbers as a generalization of $I$-statistical convergence.
They proved that every $I$-statistically convergent sequence is
$I$-statistically pre-Cauchy and the converse is true under
certain sufficient conditions.

Following the line of Das and Savas [6], in this paper we
introduce the notion of strong $I$-statistically pre-Cauchy double
sequences in probabilistic metric spaces. We also introduce the
new notion of strong $I^*$-statistically pre-Cauchy double
sequences in probabilistic metric spaces. We show that every
strong $I$-statistically convergent double sequence is strong
$I$-statistically pre-Cauchy and every strong $I^*$-statistically
pre-Cauchy double sequence is strong $I$-statistically pre-Cauchy
and the converse of each of the results holds under certain
conditions.

\section{\textbf{Preliminaries  }}
 We now recall some definition and notation
\begin{defn}
If $K$ is a subset of the set of positive integer $\mathbb N,$ then $K_n$ denotes the set $\left\{k\in K:k\leq n\right\} $and $\left|K_n\right|$ denotes the number of elements in $K_n$.
 The ``natural density"of $K$
[4] is given by $\delta(K)=\displaystyle{\lim_{ n \rightarrow \infty}}{ \left|K_n\right|/n}$.
\end{defn}
\begin{defn}$[25]$
 Let $X\neq\phi $. A class $ I \subseteq 2^X $ of subsets of X is said to be
an ideal in X provided that $I$ satisfies these conditions:
\\(i)$\phi \in I$
\\(ii)$ A,B \in I \Rightarrow A \cup B\in I,$
\\(iii)$ A \in I, B\subseteq A \Rightarrow B\in I$
\end{defn}
An ideal is called non-trivial if X $\notin I.$
\begin{defn}$[25]$
 Let $X\neq\phi  $;. A non-empty class $\mathbb F$$\subseteq 2^X $ of subsets of X is
said to be a filter in X provided that:
\\(i)$\phi\notin \mathbb F $
\\(ii) $A,B\in\mathbb F \Rightarrow A \cap B\in\mathbb F,$
\\(iii)$ A \in\mathbb F, B\supseteq A \Rightarrow B\in\mathbb F$.
\end{defn}
The following result expresses a relation between the notions of ideal and filter:
\\
\\ \textbf{Result 1}. Let $I$ be a non-trivial ideal in X, $X\neq\phi ,.$ Then the class $\mathbb F(I)$$ = \left\{M\subseteq X : \exists A \in I: M = X\setminus A\right\}$
is a filter on X (we will call $\mathbb F(I) $ the filter associated with $I$).
\\The proof of result 1 is easy and so it can be left.
\\ A non-trivial ideal $I$ in X is called admissible if $\left\{x\right\} \in I$ for each $x \in X$.
\begin{defn}$[5]$
 A non-trivial ideal $I$ of $\mathbb{N}\times\mathbb{N}$ is called strongly admissible if $\left\{i\right\}\times\mathbb N$ and $\mathbb{N}\times\left\{i\right\}$ belong to $I$ for each $i\in\mathbb N$.
\end{defn}
 Clearly every admissible ideal of $\mathbb{N}\times\mathbb{N}$ is strongly admissible.
\\
\\We can also use the concept of porosity of subsets of a metric space.
\\
\\The concept of statistical convergence and the study of similar types of convergence ([1], [2], [3], [4]) lead us to introduced the notion of $I$-convergence
of sequences. This notion gives a unifying look at many types of convergence related to statistical convergence.
\\
\begin{defn}$[25]$
 : Let $I$ be a non-trivial ideal of $\mathbb N$. A sequence $x = \left\{x_n\right\}_{n=1}^\infty$ of real numbers is said to be $I$-convergent to $\zeta\in R $ if for every $\epsilon > 0$ the set
$A(\epsilon) = \left\{n : \left|x_n -\zeta\right|\geq\epsilon\right\} \in I$.
\\
\\If $x = \left\{x_n\right\}_{n=1}^\infty $ is $I$-convergent to $\zeta$ we write $I-\displaystyle{\lim_{n \rightarrow\infty}}{ x_n }= \zeta (~or~ I-lim x = \zeta)$
and the number $\zeta$ is called the $I$-limit of $x$.
\end{defn}
We now recall some basic concepts related to statistical
convergence of double sequences (see $[7]$ for more details).
Let $K\subset\mathbb{N}\times\mathbb{N}$. Let $K(n,m)$ be the
number of $(j,k)\in K$ such that $j\leq n, k\leq m$. The number
$\overline{d}_{2}(K)=\underset{n\rightarrow\infty}{\underset{m\rightarrow\infty}{\limsup}}\frac{K(n,m)}{nm}$
is called the upper double natural density of K. If the sequence
$\{\frac{K(n,m)}{nm}\}_{n,m\in\mathbb{N}}$ has a limit in
Pringsheim's sense, then we say that $K$ has the double natural
density and it is denoted by
\begin{center}
$d_{2}(K)=\underset{n\rightarrow\infty}{\underset{m\rightarrow\infty}{\lim}}\frac{K(n,m)}{nm}$.
\end{center}
\begin{defn} $[7]$ A double sequence $x=\{x_{jk}\}_{j,k\in\mathbb{N}}$ of real numbers is said to be statistically convergent to
$\xi \in \mathbb{R}$ if for every $\epsilon> 0$, we have
$d_{2}(A(\epsilon))=0$ where $A(\epsilon)=\{(j,k)\in
\mathbb{N}\times\mathbb{N}; |x_{jk}-\xi|\geq\epsilon\}$. In
this case we write
$st-\underset{k\rightarrow\infty}{\underset{j\rightarrow\infty}{\lim}}x_{jk}=\xi$.
\end{defn}

A statistically convergent double sequence of elements in a metric space $(X,\rho)$ is defined
essentially in the same way using $\rho(x_{jk},\xi)\geq \epsilon$
instead of $|x_{jk}-\xi|\geq\epsilon$.
\begin{defn}$[7]$  Let $(X,\rho)$ be a metric space. A double sequence $x=\{x_{jk}\}_{j,k\in\mathbb{N}}$ in $X$ is said to be statistically
Cauchy if for every $\epsilon\geq 0$, there exist natural
numbers $N=N(\epsilon)$ and $M=M(\epsilon)$ such that for
all $j,p\geq N$ and $k,q\geq M$,
\begin{center}
$d_{2}(\{(j,k)\in\mathbb{N}\times\mathbb{N}: \rho(x_{jk},x_{pq})\geq\epsilon\})=0$
\end{center}
\end{defn}
\begin{thm}$[7]$
A double sequence $x=\{x_{jk}\}_{j,k\in\mathbb{N}}$ in a metric
space $(X,\rho)$ is statistically convergent to $\xi\in X $ if and
only if there exists a subset
$K=\{(j,k)\in\mathbb{N}\times\mathbb{N}\}$ of
$\mathbb{N}\times\mathbb{N}$ such that $d_{2}(K)=1$ and
$\underset{(j,k)\in
K}{\underset{k\rightarrow\infty}{\underset{j\rightarrow\infty}{\lim}}}x_{jk}=\xi$.
\end{thm}
\begin{thm}$[7]$ For a sequence $x=\{x_{jk}\}_{j,k\in\mathbb{N}}$ in a metric space $(X,\rho)$, the following statements are equivalent:
\begin{enumerate}
\item $x$ is statistically convergent to  $l\in X$. \item $x$ is
statistically Cauchy. \item There exists a subset
$M\subset\mathbb{N}\times\mathbb{N}$ such that $d_{2}(M)=1$ and
$\{x_{jk}\}_{(j,k)\in M}$ converges to $l$.
\end{enumerate}
\end{thm}
\section{\textbf{Basic concepts of Probabilistic Metric Spaces.  }}
First we recall some basic concepts related to the probabilistic metric spaces (or PM spaces) which can be studied in details from
the fundamental book [21] by Schweizer and Sklar.
\begin{defn}
A non decreasing function $F:\mathbb{R}\rightarrow[0,1]$ defined
on $\mathbb{R}$ with $F(-\infty)=0$ and $ F(\infty)=1,$ where
$\mathbb{R}=[-\infty,\infty]$ , is called a distribution
function.

The set of all left continuous distribution function over
$(-\infty,\infty)$ is denoted by $\Delta$.

We consider the relation $\leq$ on $\Delta$ defined by $G \leq F
$ if and only if $G(x)\leq F(x)$ for all $x \in \mathbb{R}$. Clearly it
can be easily verified that the relation `$\leq$' is a partial
order on $\Delta$.
\end{defn}
\begin{defn}
For any $p\in [-\infty, \infty ]$ the unit step at $p$ is denoted
by $\epsilon_{p}$ and is defined to be a function in $\Delta$
given by
\begin{eqnarray*}
\epsilon_{p}(x)&=&0, ~~ -\infty \leq x \leq p\\
                  &=&1,~~  p < x\leq \infty.
\end{eqnarray*}
\end{defn}
\begin{defn}
A sequence $\{F_{n}\}_{n\in \mathbb{N}}$ of distribution functions
converges weakly to a distribution function $F$ and we write
$F_{n}\xrightarrow{w}F$ if and only if the sequence $\{F_{n}(x)\}_{n\in \mathbb{N}}$
converges to $F(x)$ at each continuity point $x$ of $F$.
\end{defn}
\begin{defn}
The distance between F and G in $\Delta$ is denoted by
$d_{L}(F,G)$ and is defined as the infimum of all numbers
$h\in(0,1]$ such that the inequalities
\begin{center}
$~~~~~~~F(x-h)-h\leq G(x)\leq F(x+h)+h$ \\

and $G(x-h)-h\leq F(x)\leq G(x+h)+h$
\end{center}
hold for every $x\in(-\frac{1}{h},\frac{1}{h})$.
\end{defn}
It is known that $d_{L}$ is a metric on $\Delta$ and for any
sequence $\{F_{n}\}_{n\in\mathbb{N}}$ in $\Delta$ and
$F\in\Delta$, we have
\begin{center}
$F_{n}\xrightarrow{w}F$  if and only if $d_{L}(F_{n},F)\rightarrow 0$.
\end{center}

Here we will be interested in the subset of $\Delta$ consisting
of those elements $G$ that satisfy $G(0)=0 $.
\begin{defn}
A  non-decreasing function $G$
defined on $\mathbb{R}^{+}=[0,\infty]$ that satisfies $G(0)=0$
and $ G(\infty)=1$ and is left continuous on $(0,\infty)$ is called a distance distribution function.
\end{defn}

The set of all distance distribution functions is denoted by
$\Delta^{+}$.

The function $d_{L}$ is clearly a metric on $\Delta^{+}$. The
metric space $(\Delta^{+},d_{L})$ is compact and hence complete.
\begin{thm}
Let $F\in\Delta^{+}$ be given . Then for any $t>0$, $ F(t)>1-t$ if
and only if $d_{L}(F,\epsilon_{0})<t$.
\end{thm}
\begin{defn}
A triangle function is a binary operation $\tau$ on $\Delta^{+}$,
$\tau:\Delta^{+}\times\Delta^{+}\rightarrow\Delta^{+}$ which is
commutative, nondecreasing, associative in each place, and
$\epsilon_{0}$ is the identity.
\end{defn}
\begin{defn} A probabilistic metric space, briefly PM space, is a triplet $(S,\mathfrak{F},\tau)$ where S is a nonempty set whose elements
are the points of the space; $\mathfrak{F}$ is a function from $S\times S$ into $\Delta^{+}$, $\tau$ is a triangle function, and the
following conditions are satisfied for all $x,y,z\in S$:
    \begin{enumerate}
    \item $\mathfrak{F}(x,x)=\epsilon_{0}$
    \item $\mathfrak{F}(x,y)\neq\epsilon_{0}$ if $x\neq y $
    \item $\mathfrak{F}(x,y)=\mathfrak{F}(y,x) $
    \item $ \mathfrak{F}(x,z)\geq \tau(\mathfrak{F}(x,y),\mathfrak{F}(y,z))$.
    \end{enumerate}

  From now on we will denote $\mathfrak{F}(x,y)$ by $F_{xy}$ and its value at $a$ by $F_{xy}(a)$.
\end{defn}

\begin{defn}Let $(S,\mathfrak{F},\tau)$ be a PM space. For $x\in S$ and $t>0$, the strong $t$-neighbourhood of $x$ is defined as the set
    \begin{center}
    $\mathcal{N}_{x}(t)=\{y\in S : F_{xy}(t)>1-t\} $.
    \end{center}

    The collection $\mathfrak{N}_{x}=\{\mathcal{N}_{x}(t):t>0 \}$ is called the strong neighbourhood system at $x$, and the union
     $\mathfrak{N}=\underset{x\in S}{\bigcup}\mathfrak{N}_{x}$ is called the strong neighbourhood system for $ S $.

    By Theorem 2.1, we can write

    \begin{center}
    $\mathcal{N}_{x}(t)=\{y\in S : d_{L}(F_{xy},\epsilon_{0})<t\} $.
    \end{center}

    If $\tau$ is continuous, then the strong neighbourhood system $\mathfrak{N}$ determines a Hausdorff topology for S.
    This topology is called the strong topology for S.
    \end{defn}
\begin{defn}
Let $(S,\mathfrak{F},\tau)$ be a PM space. Then for any $t>0$, the
subset $\mathfrak{U}(t)$ of $ S\times S$ given by
   \begin{center}
   $\mathfrak{U}(t)=\{ (x,y): F_{xy} (t)>1-t\} $
   \end{center}
   is called the strong $t$-vicinity.
   \end{defn}

\begin{thm}
 Let $(S,\mathfrak{F},\tau)$ be a PM space and $\tau $ be continuous. Then for any $t>0$, there is an $\eta>0$ such that
    \begin{center}
    $\mathfrak{U}(\eta)\circ\mathfrak{U}(\eta)\subset \mathfrak{U}(t) $
    \end{center}
    where $\mathfrak{U}(\eta)\circ\mathfrak{U}(\eta)=\{(x,z):$ \mbox{for some} y, $ (x,y)$ and $(y,z)\in \mathfrak{U}(t)\} $.
    \end{thm}
\begin{note}
  From the hypothesis of Theorem 2.2 we can say that for any $t>0$, there is an $\eta >0 $ such that $F_{ab}(t)>1-t$
 whenever $F_{ac}(\eta)>1-\eta $ and $F_{cb}(\eta)>1-\eta$. Equivalently it can be written as: for any $t>0$, there is
 an $\eta>0$ such that $d_{L}(F_{ab},\epsilon_{0})<t$ whenever $d_{L}(F_{ac},\epsilon_{0})<\eta$ and $ d_{L}(F_{cb}, \epsilon_{0})<\eta$.
 \end{note}

    In a PM space $(S,\mathfrak{F},\tau)$, if $\tau$ is continuous then the strong neighbourhood system $\mathfrak{N}$
    determines a Kuratowski closure operation which is called the strong closure, and for any subset $A$ of $S$ the strong closure of $A$ is denoted by
    $\kappa(A)$  and for any nonempty subset $A$ of $S$
    \begin{center}
    $ \kappa(A)=\{ x\in S:$ for any $t>0$, there is a $y \in A$ such that $F_{xy}(t)>1-t \} $.
    \end{center}

\begin{rem} Throughout the rest of the article, in a PM space $(S,\mathfrak{F},\tau)$, we always assume that $\tau$ is
continuous and $S$ is endowed with the strong topology.
\end{rem}

\begin{defn}
Let $(S,\mathfrak{F},\tau)$ be a PM space. A sequence $\{
x_{n}\}_{n\in\mathbb{N}}$ in $S$ is said to be strong convergent
to a point $x \in S $ if for any $t>0$ , there exists a natural
number $N$ such that $x_{n}\in\mathcal{N}_{x}(t)$ where $n\geq
N$ and we write $x_{n}\rightarrow x$ or
$\underset{n\rightarrow\infty}{\lim}x_{n}=x$.

Similarly a sequence $\{ x_{n}\}_{n\in\mathbb{N}}$ in $S$ is called a strong Cauchy sequence if for any $t>0$,
there exists a natural number $N $ such that $(x_{m},x_{n})\in\mathfrak{U}(t)$ whenever $m,n \geq N$.

By the convergence of a double sequence we  mean convergence in
Pringsheim's sense $[16]$. A real double sequence
$x=\{x_{jk}\}_{j,k\in \mathbb{N}}$ is said to converge to a real
number $a$ if for every $\epsilon>0$, there exists a $N\in
\mathbb{N}$ such that $|x_{jk}-a|< \epsilon$ whenever
$j,k\geq N$.

A real double sequence $x=\{x_{jk}\}_{j,k\in\mathbb{N}}$ is said to be a Cauchy Sequence if for every
$\epsilon>0$, there exists $N,M\in\mathbb{N}$ such that for all $j,p\geq N$ ;  $k,q\geq M$, $|x_{jk}-x_{pq}|<\epsilon$.
\end{defn}
\begin{defn}
Let $(S,\mathfrak{F},\tau)$ be a PM space. A double sequence $x=\{x_{jk}\}_{j,k\in\mathbb{N}}$ in $S$ is said to be strong
convergent to a point $\xi\in S$ if for any $t>0$, there exists a natural number $K$ such that $x_{jk}\in \mathcal{N}_{\xi}(t)$ whenever $j,k\geq K $.

In this case we write $x_{jk}\rightarrow \xi$ or
$\underset{k\rightarrow\infty}{\underset{j\rightarrow\infty}{\lim}}x_{jk}=\xi$.

Similarly a double sequence $x=\{ x_{jk}\}_{j,k \in \mathbb{N}}$
in S is called a strong Cauchy double sequence if for any $t>0$,
there exist natural numbers $N,M$ such that for all $j,p\geq N$;
$k,q\geq M $, $(x_{jk}, x_{pq})\in \mathfrak{U}(t)$.
\end{defn}



\section{\textbf{Main Results  }}
In this section, we are concerned with ideal statistical pre-Cauchy and ideal statistical convergence for double sequences in a PM space.
\begin{defn}$[11]$
 A double sequence $\left(x_{jk}\right)_{j,k\in \mathbb{N}}$ of real numbers is $I$-statistically convergent to a real numbers $L$, and we write $x_{jk}\stackrel{I^s}{\rightarrow}L$, provided that for $\epsilon> 0$ and $\delta> 0$
\begin{center}
$\left\{\left(m,n\right)\in\mathbb{N}\times\mathbb{N} :\frac{1}{mn}\left|\left\{\left(j,k\right):\left|x_{jk}-L\right|\geq\epsilon,j\leq m,k\leq n\right\}\right|\geq\delta\right\}\in I$
\end{center}
\end{defn}
\begin{defn}$[11]$
  A double sequence $\left(x_{jk}\right)_{j,k\in \mathbb{N}}$ of real numbers is said to be $I$-statistically pre-Cauchy if for any $\epsilon> 0 $ and $\delta> 0$
\begin{center}
$\left\{(m,n)\in\mathbb{N}\times\mathbb{N}:\frac{1}{m^2n^2}\left|\left\{(j,k):\left|x_{jk}-x_{pq}\right|\geq\epsilon;j,p\leq m;k,q\leq n\right\}\right|\geq\delta\right\}\in I$.
\end{center}
\end{defn}
 Now we introduce some definition in a Probabilistic Metric space.
\begin{defn}
Let $\left(S,\mathfrak{ F},\tau\right)$ be a PM space. A double sequence $\left(x_{jk}\right)_{j,k\in \mathbb{N}}$ in S is strong $I$-statistically convergent to p in S, and we write $x_{jk}\stackrel{str-I^s}{\rightarrow}p$, provided that for $t> 0$ and $\delta> 0$
\begin{center}
$\left\{\left(m,n\right)\in\mathbb{N}\times\mathbb{N} :\frac{1}{mn}\left|\left\{\left(j,k\right):x_{jk}\notin\mathcal{N}_p(t),j\leq m,k\leq n\right\}\right|\geq\delta\right\}\in I$
\end{center}
\end{defn}
\begin{defn}
 Let $\left(S,\mathfrak{ F},\tau\right)$ be a PM space. A double sequence $\left(x_{jk}\right)_{j,k\in \mathbb{N}}$ in S is said to be strong $I$-statistically pre-Cauchy if for any $t> 0 $ and $\delta> 0$
\begin{center}
$\left\{(m,n)\in\mathbb{N}\times\mathbb{N}:\frac{1}{m^2n^2}\left|\left\{(j,k):(x_{jk},x_{pq})\notin \mathfrak{U}(t);j,p\leq m;k,q\leq n\right\}\right|\geq\delta\right\}\in I$.
\end{center}
\end{defn}
\begin{note}
In a PM space a double sequence which is strong statistically
convergent is clearly strong $I$-statistically convergent but
converse is not true for this we consider the following example.
Consider the equilateral PM space $(S,\mathfrak{F},\tau)$ where
$\mathfrak{F}$ is defined by \[ \mathfrak{F}_{pq} = \left\{
  \begin{array}{l l}
    F & ;\quad \text{if $ p \neq q $ }\\
    \epsilon_0 &; \quad \text{if $p = q$ }
  \end{array} \right.\] and M is the maximal triangular
function. Here $F\in\Delta^+$ is fixed and distinct from
$\epsilon_0$ and $\epsilon_\infty$. Let $I$ is an ideal of
$\mathbb{N}\times\mathbb{N}$. We let $A\in I$ such that
$A=\left\{(t_m,t_n):(m,n)\in\mathbb{N}\times\mathbb{N}\right\}$.
Let B be an subset of $\mathbb{N}\times\mathbb{N}$ such that
$d_2(B)=0$.Now fora fixed $p,q\in S$ we define
\[ x_{t_mt_n} = \left\{
  \begin{array}{l l}
    p & ;\quad \text{if $ (m,n) \in B $ }\\
    q &; \quad \text{if $ (m,n)\notin B $ }
  \end{array} \right.\]
 and $ x_{mn}= p ~ if ~(m,n)\notin A $.

Then $\left\{x_{mn}\right\}_{m,n\in\mathbb{N}}$ is $I$-statistically convergent to p but is not statistically convergent.
\end{note}
\begin{thm}
 An strong $I$-statistically convergent double sequence is strong $I$-statistically pre-Cauchy double sequence in a PM space.
\end{thm}
 \begin{proof}
Let $ \left\{x_{jk}\right\}_{j,k\in \mathbb{N}}$ be strong $I$-statistically convergent to a in $S$. Let $t>0$ and $\delta>0$. Choose $\delta_1>0$ such that $1-(1-\delta_1)^2<\delta$.
\begin{center}
 Now for that $t> 0$, there exists an $\eta>0$ such that for all $a,b,c \in S$ we have $d_L(F_{ac},\epsilon_0)<t$ whenever $d_L(F_{ab},\epsilon_0)<\eta$ and $d_L(F_{bc},\epsilon_0)<\eta$..................(1).
\end{center}
Now for $\delta_1 > 0$ and $\eta> 0$ we have
\begin{center}
 Let $C=\left\{\left(m,n\right)\in \mathbb{N}\times\mathbb{N} :\frac{1}{mn} \left|\left\{\left(j,k\right):x_{jk}\notin\mathcal{ N}_a(\eta),j\leq m,k\leq n\right\}\right|\geq\delta_1\right\}\in I.$

\end{center}
 Let $(m,n)\in C^c$ then
\begin{center}
$\frac{1}{mn} \left|\left\{\left(j,k\right):x_{jk}\notin\mathcal{N}_a(\eta),j\leq m,k\leq n\right\}\right|<\delta_1$
\end{center}

\begin{center}
$\Rightarrow \frac{1}{mn}\left|\left\{\left(j,k\right):x_{jk}\in \mathcal{N}_a(\eta),j\leq m,k\leq n\right\}\right|>{1-\delta_1}.$

\end{center}
Let $B_{mn}=\left\{\left(j,k\right):x_{jk}\in \mathcal{N}_a(\eta),j\leq m,k\leq n\right\}$.
\\ Then for $(j,k), (p,q)\in B_{mn}$,
 $d_L(F_{x_{jk}a},\epsilon_0)<\eta$ and $d_L(F_{x_{pq}a},\epsilon_0)<\eta$ which implies $d_L(F_{x_{jk}x_{pq}},\epsilon_0)<t$ from (1).
\\ This implies $[\left|B_{mn}\right|^2/m^2n^2]\leq \frac{1}{m^2n^2}\left|\left\{(j,k):(x_{jk},x_{pq})\in\mathfrak{ U}(t);j,p\leq m;k,q\leq n\right\}\right|  $. Thus for all $(m,n)\in C^c$ we have
\begin{center}
$(1-\delta_1)^2<[\left|B_{mn}\right|^2/m^2n^2]\leq \frac{1}{m^2n^2}\left|\left\{(j,k):(x_{jk},x_{pq})\in\mathfrak{U}(t);j,p\leq m;k,q\leq n\right\}\right|  $

\end{center}
$\Rightarrow \frac{1}{m^2n^2}\left|\left\{(j,k):(x_{jk},x_{pq})\notin \mathfrak{U}(t);j,p\leq m;k,q\leq n\right\}\right|\leq 1-(1-\delta_1)^2<\delta  $. We see that for all $(m,n)\in C^c$, $ \frac{1}{m^2n^2}\left|\left\{(j,k):(x_{jk},x_{pq})\notin\mathfrak{U}(t);j,p\leq m;k,q\leq n\right\}\right|<\delta  $ and so
\begin{center}
$\left\{(m,n)\in\mathbb{N}\times\mathbb{N} : \frac{1}{m^2n^2}\left|\left\{(j,k):(x_{jk},x_{pq})\notin\mathfrak{U}(t);j,p\leq m;k,q\leq n\right\}\right|\geq\delta\right\}\subset C $.
\end{center}
 Since $C\in I$,
\begin{center}
$\left\{(m,n)\in\mathbb{N}\times\mathbb{N} : \frac{1}{m^2n^2}\left|\left\{(j,k):(x_{jk},x_{pq})\notin\mathfrak{U}(t);j,p\leq m;k,q\leq n\right\}\right|\geq\delta\right\}\in I.$
\end{center}
Hence $x$ is strong $I$-statistically pre-Cauchy.
\end{proof}
The next result gives a necessary and sufficient condition for a double sequence to be $I$-statistically pre-Cauchy.
\\
\begin{thm}
 Let $x = \left\{x_{jk}\right\}$ be double sequence in a PM space $(S,\mathfrak{ F},\tau)$. A double sequence $x = \left\{x_{jk}\right\} $ is  strong $ I$-statistically pre-
Cauchy if and only if $I-\lim\limits_{\stackrel{\stackrel{m\rightarrow\infty}{n\rightarrow\infty}}~} $$\frac{1}{m^2n^2} \sum\limits_{\stackrel{j,p\leq m}~}\sum\limits_{\stackrel{k,q\leq n}~}d_L(F_{{x_{jk}}{x_{pq}}},\epsilon_0)=0$
\end{thm}
\begin{proof}
 First we assume that $I-\lim\limits_{\stackrel{\stackrel{m\rightarrow\infty}{n\rightarrow\infty}}~} $$\frac{1}{m^2n^2} \sum\limits_{\stackrel{j,p\leq m}~}\sum\limits_{\stackrel{k,q\leq n}~}d_L(F_{{x_{jk}}{x_{pq}}},\epsilon_0)=0$. Note that for $t>0$ and $(m,n)\in \mathbb{N}\times\mathbb{N}$ we have $\frac{1}{m^2n^2} \sum\limits_{\stackrel{j,p\leq m}~}\sum\limits_{\stackrel{k,q\leq n}~}d_L(F_{{x_{jk}}{x_{pq}}},\epsilon_0) \geq t(\frac{1}{m^2n^2}\left|\left\{(j,k):d_L(F_{{x_{jk}}{x_{pq}}},\epsilon_0)\geq t;j,p\leq m;k,q\leq n\right\}\right|)$.
\\ Therefore for any $\delta>0$,\\ $A=\left\{(m,n)\in \mathbb{N}\times\mathbb{N}:\frac{1}{m^2n^2}\left|\left\{(j,k):d_L(F_{{x_{jk}}{x_{pq}}},\epsilon_0)\geq t;j,p\leq m;k,q\leq n\right\}\right|\geq\delta\right\}$\\$\subset \left\{(m,n)\in \mathbb{N}\times\mathbb{N} :\frac{1}{m^2n^2} \sum\limits_{\stackrel{j,p\leq m}~}\sum\limits_{\stackrel{k,q\leq n}~}d_L(F_{{x_{jk}}{x_{pq}}},\epsilon_0)\geq\delta t\right\}$.\\ Since$I-\lim\limits_{\stackrel{\stackrel{m\rightarrow\infty}{n\rightarrow\infty}}~} $$\frac{1}{m^2n^2} \sum\limits_{\stackrel{j,p\leq m}~}\sum\limits_{\stackrel{k,q\leq n}~}d_L(F_{{x_{jk}}{x_{pq}}},\epsilon_0)=0$ thus right hand side belongs to $I$ which implies that
\begin{center}
$\left\{(m,n)\in\mathbb{N}\times\mathbb{N}:\frac{1}{m^2n^2}\left|\left\{(j,k):d_L(F_{{x_{jk}}{x_{pq}}},\epsilon_0)\geq t;j,p\leq m;k,q\leq n\right\}\right|\geq\delta\right\}\in I.$
\end{center}
 This shows that $x$ is strong $I$-statistically pre-Cauchy.
\\
\\   Conversely assume that $x$ is strong $I$- statistically pre-Cauchy double sequence in a PM space and  let $\delta>0$ has been given. Choose $t>0$ and $\delta_1>0$ such that $\frac{t}{2}+\delta_1<\delta$. Then for each $(m,n)\in \mathbb{N}\times\mathbb{N}$,
\\ $\frac{1}{m^2n^2} \sum\limits_{\stackrel{j,p\leq m}~}\sum\limits_{\stackrel{k,q\leq n}~}d_L(F_{{x_{jk}}{x_{pq}}},\epsilon_0)$
\\= $\frac{1}{m^2n^2} \sum\limits_{\stackrel{\stackrel{d_L(F_{{x_{jk}}{x_{pq}}},\epsilon_0)<\frac{t}{2}}{j,p\leq m,k,q\leq n}}~}d_L(F_{{x_{jk}}{x_{pq}}},\epsilon_0)$+
 $\frac{1}{m^2n^2} \sum\limits_{\stackrel{\stackrel{d_L(F_{{x_{jk}}{x_{pq}}},\epsilon_0)\geq\frac{t}{2}}{j,p\leq m,k,q\leq n}}~}d_L(F_{{x_{jk}}{x_{pq}}},\epsilon_0)$
$\leq \frac{t}{2}+\left\{\frac{1}{m^2n^2}\left|\left\{(j,k):d_L(F_{{x_{jk}}{x_{pq}}},\epsilon_0)\geq\frac{t}{2};j,p\leq m;k,q\leq n\right\}\right|\right\}$. [ Since $d_L(F_{x_{jk}x_{pq}},\epsilon_0)\leq 1$].
\\ Now since $x$ is strong $I$-statistically pre-Cauchy for that $\delta_1>0$
\begin{center}
$A= \left\{(m,n)\in\mathbb{N}\times\mathbb{N}:\frac{1}{m^2n^2}\left|\left\{(j,k):d_L(F_{{x_{jk}}{x_{pq}}},\epsilon_0)\geq \frac {t}{2};j,p\leq m;k,q\leq n\right\}\right|\geq\delta_1\right\}\in I$
\end{center}
Then for $(m,n)\in A^c$,  $\frac{1}{m^2n^2}\left|\left\{(j,k):d_L(F_{{x_{jk}}{x_{pq}}},\epsilon_0)\geq t;j,p\leq m;k,q\leq n\right\}\right|<\delta_1$. So $\frac{1}{m^2n^2} \sum\limits_{\stackrel{j,p\leq m}~}\sum\limits_{\stackrel{k,q\leq n}~}d_L(F_{{x_{jk}}{x_{pq}}},\epsilon_0)\leq\frac{t}{2}+\delta_1<\delta$. We see that for all $(m,n)\in A^c$ we have $\frac{1}{m^2n^2} \sum\limits_{\stackrel{j,p\leq m}~}\sum\limits_{\stackrel{k,q\leq n}~}d_L(F_{{x_{jk}}{x_{pq}}},\epsilon_0)<\delta$. That is
\begin{center}
$\left\{(m,n)\in \mathbb{N}\times\mathbb{N}:\frac{1}{m^2n^2} \sum\limits_{\stackrel{j,p\leq m}~}\sum\limits_{\stackrel{k,q\leq n}~}d_L(F_{{x_{jk}}{x_{pq}}},\epsilon_0)\geq\delta_1\right\}\subset A$.
\end{center}
  Since $A\in I$ so
    \begin{center}
    $\left\{(m,n)\in\mathbb{N}\times\mathbb{N}:\frac{1}{m^2n^2} \sum\limits_{\stackrel{j,p\leq m}~}\sum\limits_{\stackrel{k,q\leq n}~}d_L(F_{{x_{jk}}{x_{pq}}},\epsilon_0)\geq\delta_1\right\}\in I$.
    \end{center}
     Therefore $I-\lim\limits_{\stackrel{\stackrel{m\rightarrow\infty}{n\rightarrow\infty}}~} $$\frac{1}{m^2n^2} \sum\limits_{\stackrel{j,p\leq m}~}\sum\limits_{\stackrel{k,q\leq n}~}d_L(F_{{x_{jk}}{x_{pq}}},\epsilon_0)=0$. This proves the necessity of the theorem.
\end{proof}
 Now we give a sufficient condition under which a strong $I$-statistically pre-Cauchy double sequence can be a strong $I$-statistically convergent.
\begin{defn}$[11]$
Let $I$ be an admissible ideal of $\mathbb{N}\times\mathbb{N}$ and $x=\left\{x_{jk}\right\}_{j,k\in\mathbb{N}}$
be a real double sequence. Let $A_x=\left\{\alpha\in\mathbb{R}:\left\{(j,k):x_{jk}<\alpha\right\}\notin I\right\}$.
Then $I$-limit inferior of $x$ is given by
\[ I-\liminf x= \left\{
  \begin{array}{l l}
    \inf A_x & ;\quad \text{if $A_x \neq \phi$ }\\
    \infty &; \quad \text{if $A_x = \phi$ }
  \end{array} \right.\]\\

 It is known (Theorem 3, [28]) that $I$-lim inf $x =\alpha $ (finite) if and only if for arbitrary $\epsilon>0$,
\begin{center}
$\left\{(j,k):x_{jk}<\epsilon+\alpha\right\}\notin I$ and $\left\{(j,k):x_{jk}<\alpha-\epsilon\right\}\in I$
\end{center}
\end{defn}
\begin{defn}
Let $x=\left\{x_{ij}\right\}_{i,j\in\mathbb{N}}$ be a  double sequence of a PM space $(S,\mathfrak{ F},\tau)$. Let K be a subset of $\mathbb{N}\times\mathbb{N}$ such that for each $( i,j)\in \mathbb{N}\times\mathbb{N},$ there exists a $(m,n)\in K$ such that $(m,n)>(i,j)$ with respect to dictionary order. Then we define $\left\{x\right\}_K = \left\{x_{ij}\right\}_{i,j\in K}$ as a subsequence of $x$.
\end{defn}
\begin{thm}
:  Let $(S,\mathfrak{ F},\tau)$ be a PM space and $x=\left\{x_{jk}\right\}_{j,k\in\mathbb N}$ be  a strong $I$-statistically pre-Cauchy double sequence in $S$. If $x=\left\{x_{jk}\right\}_{j,k\in\mathbb{N}}$ has a subsequence $\left\{x_{t_jt_k}\right\}_{j,k\in \mathbb N}$ which is strong converges to $L$ and
\begin{center}
$0<I-\lim\limits_{\stackrel{\stackrel{m\rightarrow\infty}{n\rightarrow\infty}}~}inf\frac{1}{mn}\left|\left\{(t_j,t_k):t_j\leq m,t_k\leq n;j,k\in \mathbb N\right\}\right|<\infty$,
\end{center}
 then $x$ is strong $I$-statistically convergent to $L$.
\end{thm}
\begin{proof}
 Let $I-\lim\limits_{\stackrel{\stackrel{m\rightarrow\infty}{n\rightarrow\infty}}~}inf\frac{1}{mn}\left|\left\{(t_j,t_k):t_j\leq m,t_k\leq n;j,k\in \mathbb N\right\}\right|=r.$ Then $0<r<\infty$. Let $t>0$  and $\delta>0$ be given. We choose $\delta_1>0$ such that $\frac{2\delta_1}{r}<\delta$. Now for that $t>0$ there exists $\eta>0$ such that for all $a,b,c \in S$ we have $d_L(F_{ac},\epsilon_0)<t$ whenever $d_L(F_{ab},\epsilon_0)<\eta$ and $d_L(F_{bc},\epsilon_0)<\eta$..................(1). Now select $n_0\in\mathbb N$ such that $t_j>n_0$, $t_k>n_0$ for some $j,k\in\mathbb N$ then $d_L(F_{x_{t_jt_k}L},\epsilon_0)<\eta$. Let $A=\left\{(t_j,t_k):t_j>n_0, t_k>n_0;j,k\in \mathbb{N}\right\}$ and $B(t)=\left\{(j,k):d_L(F_{x_{jk}L},\epsilon_0)\geq t\right\}$. But we have from (1) that
\\ $\frac{1}{m^2n^2}\left|\left\{(j,k):d_L(F_{x_{jk}L},\epsilon_0)\geq\eta;j,p\leq m;k,q\leq n\right\}\right|\geq\frac{1}{m^2n^2}\sum\limits_{\stackrel{j,p\leq m}~}\sum\limits_{\stackrel{k,q\leq n}~}\chi_{A\times B(t)}((j,k)\times (p,q))$\\=$\frac{1}{mn}\left|\left\{(t_j,t_k)\in A:t_j\leq  m,t_k\leq n\right\}\right|\times\frac{1}{mn}\left|\left\{(p,q):d_L(F_{x_{pq}L},\epsilon_0)\geq t;p\leq m,q\leq n\right\}\right|$. Since $x$ is strong $I$-statistically pre-Cauchy so for $\delta_1>0$ and $\eta>0$
\begin{center}
$C=\left\{(m,n)\in\mathbb{N}\times\mathbb{N}:\frac{1}{m^2n^2}\left|\left\{(j,k):d_L(F_{x_{jk}{x_{pq}}},\epsilon_0)\geq\eta;j,p\leq m;k,q\leq n\right\}\right|\geq\delta_1\right\}\in I$.
\end{center}
Therefore for every $(m,n)\in C^c$ we have $\frac{1}{m^2n^2}\left|\left\{(j,k):d_L(F_{x_{jk}{x_{pq}}},\epsilon_0)\geq\eta\right\}\right|<\delta_1.......(ii)$. Now since $I-\lim\limits_{\stackrel{\stackrel{m\rightarrow\infty}{n\rightarrow\infty}}~}inf\frac{1}{mn}\left|\left\{(t_j,t_k):t_j\leq m,t_k\leq n;j,k\in \mathbb N\right\}\right|=r$, so the set$D=\left\{(m,n):\frac{1}{mn}\left|\left\{(t_j,t_k):t_j\leq m,t_k\leq n;j,k\in \mathbb N\right\}\right|<\frac{r}{2}\right\}\in I$. So every $(m,n)\in D^c$ we have $\frac{1}{mn}\left|\left\{(t_j,t_k):t_j\leq m,t_k\leq n;j,k\in \mathbb N\right\}\right|\geq\frac{r}{2}...............(iii)$. Now from (ii), (iii) we get for every $(m,n)\in C^c\cup D^c=(C\cup D)^c ,$
\begin{center}
 $\frac{1}{mn}\left|\left\{(p,q):d_L(F_{x_{pq}L},\epsilon_0)\geq t\right\}\right|<\frac{2\delta_1}{r}<\delta$.

\end{center}
 This implies $\left\{(m,n)\in\mathbb{N}\times\mathbb{N}:\frac{1}{mn}\left|\left\{(j,k):d_L(F_{x_{jk}L},\epsilon_0)\geq t;j\leq m,k\leq n\right\}\right|\geq\delta\right\}\subseteq (C\cup D)$. Since $C,D\in I$ thus $C\cup D\in I$ and so
 \begin{center}
 $\left\{(m,n)\in \mathbb{N}\times\mathbb{N}:\frac{1}{mn}\left|\left\{(j,k):d_L(F_{x_{jk}L},\epsilon_0)\geq t;j\leq m,k\leq n\right\}\right|\geq\delta\right\}\in I.$
 \end{center}
    This shows that $x$ is strong $I$-statistically convergent to $L$.
\end{proof}
 To give an example of a sequence which is strong $I$-statistically pre-Cauchy but not strong $I$-statistically convergent we first observe that every strong $I$-statistically convergent sequence must have a strong convergent subsequence which is convergent in the usual sense. Since it is not straight forward so we give proof below.
\\ Let $x=\left\{x_{jk}\right\}$ be a strong $I$-statistically convergent double sequence in a PM space convergent to a. For $\delta=t=1$ we have
\begin{center}
 $C=\left\{\left(m,n\right)\in \mathbb{N}\times\mathbb{N} :\frac{1}{mn} \left|\left\{\left(j,k\right):d_L(F_{x_{jk}a},\epsilon_0)\geq 1;j\leq m,k\leq n\right\}\right|\geq 1\right\}\in I.$
\end{center}
Since $I$ is an non-trivial ideal of $\mathbb{N}\times\mathbb{N}$ so $C\neq \mathbb{N}\times\mathbb{N}$ thus there exists $(m_1,n_1)\in C^c$ so that
\begin{center}
$\frac{1}{m_1n_1} \left|\left\{\left(j,k\right):d_L(F_{x_{jk}a},\epsilon_0)\geq 1;j\leq m_1,k\leq n_1\right\}\right|< 1$
$\Rightarrow \frac{1}{m_1n_1} \left|\left\{\left(j,k\right):d_L(F_{x_{jk}a},\epsilon_0)< 1;j\leq m_1,k\leq n_1\right\}\right|>0.$
\end{center}
So there exists $j_1\leq m_1$ and $k_1\leq n_1$ such that $d_L(F_{x_{j_1k_1}a},\epsilon_0)< 1$. Again taking $\delta=t=\frac{1}{2}$ we have
\begin{center}
$D=\left\{\left(m,n\right)\in \mathbb{N}\times\mathbb{N} :\frac{1}{mn} \left|\left\{\left(j,k\right):d_L(F_{x_{jk}a},\epsilon_0)\geq\frac{1}{2};j\leq m,k\leq n\right\}\right|\geq \frac{1}{2}\right\}\in I.$
\end{center}
Let $I$ is strongly admissible ideal then we have $D\cup(\mathbb{N}\times\left\{1,2,...,4n_1\right\})\cup(\left\{1,2,...,4m_1\right\}\times\mathbb N)\in I.$ Since $I$ is non-trivial choose $(m_2,n_2)$ such that $(m_2,n_2)\notin D$ and $m_2>4m_1,n_2>4n_1$. Then
\begin{center}
$\frac{1}{m_2n_2} \left|\left\{(j,k):d_L(F_{x_{jk}a},\epsilon_0)\geq\frac{1}{2};j\leq m_2,k\leq n_2\right\}\right|<\frac{1}{2}$
$\Rightarrow \frac{1}{m_2n_2} \left|\left\{\left(j,k\right):d_L(F_{x_{jk}a},\epsilon_0)<\frac{1}{2};j\leq m,k\leq n\right\}\right|> \frac{1}{2}$.
\end{center}
Now that if $d_L(F_{x_{jk}a},\epsilon_0)\geq\frac{1}{2}$ for all $m_1<j<m_2$ and for all $n_1<k<n_2$. Then $\frac{1}{m_2n_2} \left|\left\{\left(j,k\right):d_L(F_{x_{jk}a},\epsilon_0)<\frac{1}{2};j\leq m,k\leq n\right\}\right|\leq \frac{m_1n_1}{m_2n_2}<\frac{1}{16}.$ Consequently there exist $m_1<j\leq m_2$ and $n_1<k\leq n_2$ such that $d_L(F_{x_{jk}a},\epsilon_0)<\frac{1}{2}$. Now we write $j=j_2$ and $k=k_2$ then clearly $j_1<j_2$ and $k_1<k_2$. Proceeding in this way we get a set $K=\left\{(j_1,k_1),(j_2,k_2),...\right\}$ with
$j_1<j_2$, $k_1<k_2$ and $d_L(F_{x_{j_ik_i}a},\epsilon_0)<\frac{1}{i}$. This shows that $x$ has a subsequence $\left\{x\right\}_K$ which is strong convergent to a.
\\We now construct the following example
\\
\\ \textbf{Example 1}: Let $(S,d)$ be the Euclidean line and $H(x)=1-{e^{-x}}$ where $H\in\Delta^+$. Consider the simple space $(S,d,H)$ which is generated by $(S,d)$ and $H$. Then this space becomes a PM space $(S,\mathfrak{F})$ under the continuous triangle function $\tau_M$, which is in fact a Menger space, where $\mathfrak{F}$ is defined on $S\times S$ by

\begin{center}
$\mathfrak{F}(p,q)(t)=F_{pq}(t)=H(\frac{t}{d(p,q)})=1-e^{\frac{-t}{\left|p-q\right|}}$ for all $p,q\in S$ and $t\in\mathbb{R^+}.$
\end{center}
  Here we make the convention that $H(t/0)=H(\infty)=1$ for $t>0$, and $H(0)=H(\frac{0}{0})=0$. Now let $x=\left\{x_{jk}\right\}$ be a double sequence in the PM space $S$ defined by
\begin{center}
$x_{jk}=\sum\limits_{\stackrel{u=1}~}^{m}{\frac{1}{u}}+\sum\limits_{\stackrel{v=1}~}^{n}{\frac{1}{v}}.$
\end{center}
Where $(m-1)!<j\leq m!$ and $(n-1)!<k\leq n!$. Clearly $x$ has no strong convergent subsequence by construction of the sequence. But $x$ is strong statistically pre-Cauchy since for given $t_1>0$ we have the following we define $G_m(t)=1-e^{\frac{-t}{\frac{2}{m}}}$ then clearly $G_m(t)$ is a d.d.f and for $t>0$, $G_m(t)$ weakly convergence to $\epsilon_0$ as $m\rightarrow\infty$. So for that $t_1>0$ there exists an positive integer $M$ such that for all $m\geq M$ we have $d_L(G_m(t),\epsilon_0)<t_1$. Choose $m>M$ and  $m<n$ then for  $m!<m_1\leq (m+1)!$ and $n!<n_1\leq (n+1)!$, also we have $(m-1)!<j,p\leq m_1$, $(n-1)!<k,q\leq n_1$ then $\left|x_{jk}-x_{pq}\right|<\frac{2}{m}$. It follows that for $t_1>0$ and $m!<m_1\leq(m+1)!$, $n!<n_1\leq(n+1)!$.
\begin{center}
$\frac{1}{m_1^2n_1^2}\left|\left\{(j,k):d_L(F_{x_{jk}x_{pq},\epsilon_0})<t; j,p\leq m_1,k,q\leq n_1\right\}\right|\geq \frac{1}{m_1^2n_1^2}[m_1-(m-1)!]^2[n_1-(n-1)!]^2\geq[1-\frac{1}{m}]^2[1-\frac{1}{n}]^2$. Since $\lim\limits_{\stackrel{m\rightarrow\infty,n\rightarrow\infty}~}[1-\frac{1}{m}]^2[1-\frac{1}{n}]^2=1$,
\end{center}
it follows that $x$ is strong statistically pre-Cauchy hence strong $I$-statistically pre-Cauchy.
\begin{defn}
 : Let $(S,\mathfrak{ F},\tau)$ be a PM space and $I$ be a strongly admissible ideal of $\mathbb{N}\times\mathbb{N}$. A double sequence $x=\left\{x_{jk}\right\}_{j,k\in \mathbb{N}}$ in S is said to be strong $I^*$-statistically pre-Cauchy if there exists a set $M\in \mathbb F(I)$ such that $\left\{x\right\}_M$ is strong statistically pre-Cauchy i.e; for $t>0$
\begin{center}
$\lim\limits_{\stackrel{\stackrel{m,n\rightarrow\infty}{(m,n)\in M}}~}\frac{1}{m^2n^2}\left|\left\{(j,k):d_L(F_{x_{jk}x_{pq}},\epsilon_0)\geq t;j,p\leq m;k,q\leq n\right\}\right|=0.$

\end{center}
\end{defn}
\begin{thm}
: Let $(S,\mathfrak{ F}, \tau)$ be a PM space and $x=\left\{x_{jk}\right\}_{j,k\in \mathbb{N}}$ which is strong $I^*$-statistically pre-Cauchy double sequence in $S$ then $x$ is strong $I$-statistically pre-Cauchy.
\end{thm}
\begin{proof}
Let $t>0$ and $\delta>0$ be given. Since $x=\left\{x_{jk}\right\}_{j,k}$ is strong $I^*$-statistically pre-Cauchy so there exists a set $M\in\mathbb{F}(I)$ such that
\begin{center}
$\lim\limits_{\stackrel{\stackrel{m,n\rightarrow\infty}{(m,n)\in M}}~}\frac{1}{m^2n^2}\left|\left\{(j,k):d_L(F_{x_{jk}x_{pq}},\epsilon_0)\geq t;j,p\leq m;k,q\leq n\right\}\right|=0.$

\end{center}
Then there exists $n_0\in \mathbb{N}$ such that for $(m,n)\in M$ with $m\geq n_0$, $n\geq n_0$ we have
$\frac{1}{m^2n^2}\left|\left\{(j,k):d_L(F_{x_{jk}x_{pq}},\epsilon_0)\geq t;j,p\leq m;k,q\leq n\right\}\right|<\delta.$ Let $K=\left\{1,2,....,n_0-1\right\}.$ Then obviously
\\ $A=\left\{(m,n)\in \mathbb{N}\times\mathbb{N}: \frac{1}{m^2n^2}\left|\left\{(j,k):d_L(F_{x_{jk}x_{pq}},\epsilon_0)\geq t;j,p\leq m;k,q\leq n\right\}\right|\geq \delta\right\}\subseteq (\mathbb{N}\times\mathbb{N}\setminus M )\cup(K\times\mathbb N)\cup (\mathbb{N}\times K).........(i)$. Since $I$ is strongly admissible ideal and $A\in I$ so the set on the right side of (i) belongs to $I$. Hence $x=\left\{x_{jk}\right\}_{j,k\in\mathbb{N}}$ is strong $I$-statistically pre-Cauchy double sequence in a PM space $S$.
\end{proof}
\begin{defn}$[24]$
We say an admissible ideal $I\subset 2^{\mathbb{N}\times\mathbb{N}}$ satisfies the property $(AP_2)$ if for every countable family of mutually disjoint sets $\left\{A_1,A_2,.....\right\}$ belonging to $I$, there exists a countable family of sets $\left\{B_1,B_2,.....\right\}$ such that $A_j\Delta B_j$ is included in the finite union of rows and columns in $\mathbb{N}\times\mathbb{N}$ for each $j\in\mathbb N$ and $B=\bigcup_{j=1}^{\infty}B_j\in I.$ Here $\Delta$ denotes the symmetric difference between two sets.
\end{defn}
\begin{lem} $[24]$
 Let $\left\{P_i\right\}$ be a countable collection of subsets of $\mathbb{N}\times\mathbb{N}$ such that $P_i\in \mathbb F(I)$ for each i, where $\mathbb F(I)$ is a filter associated with a strong admissible ideal $I$ with the property $(AP_2)$. Then there exists a set $P\in \mathbb{N}\times\mathbb{N}$ such that $P\in \mathbb F(I)$ and the set $P\setminus P_i$ is finite for all i.
\end{lem}
\begin{thm}
Let $(S,\mathfrak{ F},\tau)$ be a PM space. If $I$ is a strong admissible ideal of $\mathbb{N}\times\mathbb{N}$ with the property $(AP_2)$ then the notions of strong $I$-statistically pre-Cauchy and $I^*$-statistically pre-Cauchy coincide.
\end{thm}
\begin{proof}
 : From the theorem[3.4] it is sufficient to prove that a strong $I$-statistically pre-Cauchy double sequence $x=\left\{x_{jk}\right\}_{j,k\in\mathbb{N}}$ in a PM space $S$ is strong $I^*$-statistically pre-Cauchy. Let $x=\left\{x_{jk}\right\}_{j,k\in\mathbb{N}}$ in S be a strong $I$-statistically pre-Cauchy double sequence. Let $t>0$ be given. For each $i\in \mathbb{N}$, let $P_i=\left\{(m,n)\in \mathbb{N}\times\mathbb{N}:\frac{1}{m^2n^2}\left|\left\{(j,k):(x_{jk},x_{pq})\notin\mathfrak{U}(t);j,p\leq m;k,q\leq n\right\}\right|<\frac{1}{i}\right\};~$. Then $P_i\in \mathbb F(I)$ for each $i\in\mathbb N$. Since $I$ has the property $(AP_2)$, then by lemma[3.5] there exists a set $P\in \mathbb{N}\times\mathbb{N}$ such that $P\in \mathbb F(I)$ and $P\setminus P_i$ is finite for all $i\mathbb N$. Now we show that
 \begin{center}
$\lim\limits_{\stackrel{\stackrel{m,n\rightarrow\infty}{(m,n)\in P}}~}\frac{1}{m^2n^2}\left|\left\{(j,k):d_L(F_{x_{jk}x_{pq}},\epsilon_0)\geq t;j,p\leq m;k,q\leq n\right\}\right|=0.$

\end{center}
Let $\epsilon>0$ be given then there exists a $j\in \mathbb N$ such that $j>\frac{1}{\epsilon}$. Let $(m,n)\in P$ and since  $P\setminus P_j$ is a finite set, so there exists $k=k(j)\in\mathbb N$ such that $(m,n)\in P_j$ for all $m,n\geq k(j)$. Therefore for all $(m,n)\in P$ with $m,n\geq k(j)$ we have
\begin{center}
$\frac{1}{m^2n^2}\left|\left\{(j,k):d_L(F_{x_{jk}x_{pq}},\epsilon_0)\geq t;j,p\leq m;k,q\leq n\right\}\right|<\frac{1}{j}<\epsilon.$

\end{center}
Since $\epsilon>0$ is arbitrary, so
 \begin{center}
$\lim\limits_{\stackrel{\stackrel{m,n\rightarrow\infty}{(m,n)\in P}}~}\frac{1}{m^2n^2}\left|\left\{(j,k):d_L(F_{x_{jk}x_{pq}},\epsilon_0)\geq t;j,p\leq m;k,q\leq n\right\}\right|=0.$

\end{center}
Therefore $x$ is strong $I^*$-statistically pre-Cauchy.
\end{proof}
Next we present an interesting property of $I$-statistically pre-Cauchy  double sequences of real numbers in
line of Theorem 2.4 [8].
\begin{thm}
Let $x=\left\{x_{jk}\right\}_{j,k\in \mathbb N}$ be a double sequence of real numbers and $(\alpha,\beta)$ is an open interval such that $x_{jk}\notin(\alpha,\beta),$ for all $(j,k)\in \mathbb{N}\times\mathbb{N}$. We write $A=\left\{(j,k):x_{jk}\leq\alpha\right\}$ and $B=\left\{(j,k):x_{jk}\geq\beta\right\}$ and further assume that the following property is satisfied
\begin{center}
$\limsup\limits_{\stackrel{m,n\rightarrow\infty}~}D_{mn}(A)$-$\liminf\limits_{\stackrel{m,n\rightarrow\infty}~}D_{mn}(A)<r.$
\end{center}
for some $0\leq r\leq 1.$ If $x$ is $I$-statistically pre-Cauchy then either $I-\lim\limits_{\stackrel{m,n\rightarrow\infty}~}D_{mn}(A)=0$ or $I-\lim\limits_{\stackrel{m,n\rightarrow\infty}~}D_{mn}(B)=0$, where $D_{mn}(A)=\frac{1}{mn}\left|\left\{(j,k)\in A:j\leq m, k\leq n\right\}\right|$.
\end{thm}
\begin{proof}
Here $B=\mathbb{N}\times\mathbb{N}\setminus A$ and so $D_{mn}(B)=1-D_{mn}(A).$ for all $(m,n)\in\mathbb{N}\times\mathbb{N}.$ To complete the proof it is sufficient to show that either $I-\lim\limits_{\stackrel{m,n\rightarrow\infty}~}D_{mn}(A)=0$ or 1. Note that
\begin{center}
 $\chi_{A\times B}((j,k),(p,q))\leq \left|\left\{(j,k):\left|x_{jk}-x_{p,q}\right|\geq\beta-\alpha\right\}\right|$.........(1)
\end{center}
Since $x$ is $I$-statistically pre-Cauchy, so
\begin{center}
$I-\lim\limits_{\stackrel{m,n\rightarrow\infty}~}\frac{1}{mn}\left|\left\{(j,k):\left|x_{jk}-x_{p,q}\right|\geq\beta-\alpha;j,p\leq m;k,q\leq n\right\}\right|=0.$
\end{center}
But from (1) we get,
\begin{center}
$0=L.H.S= I-\lim\limits_{\stackrel{m,n\rightarrow\infty}~}D_{mn}(A)[D_{mn}(B)]=I-\lim\limits_{\stackrel{m,n\rightarrow\infty}~}D_{mn}(A)[1-D_{mn}(A)].$
\end{center}
Now from the definition of $I$-convergence it follows that
\begin{center}
$\left\{(m,n)\in\mathbb{N}\times\mathbb{N}:D_{mn}(A)[1-D_{mn}(A)]\geq\frac{1}{25}\right\}\in I.$
\end{center}
Say $M=\left\{(m,n)\in\mathbb{N}\times\mathbb{N}:D_{mn}(A)[1-D_{mn}(A)]<\frac{1}{25}\right\}\in\mathbb{F}(I).$ Clearly for all $(m,n)\in M$ either $D_{mn}(A)<\frac{1}{5}$ or $D_{mn}(A)>\frac{4}{5}$. If $D_{mn}(A)<\frac{1}{5}$ for all $(m,n)\in M_1\subset M$ for some $M_1\in \mathbb{F}(I)$ then $I-\lim\limits_{\stackrel{m,n\rightarrow\infty}~}D_{mn}(A)=0$. This is because for any $\epsilon>0$, $0<\epsilon<\frac{1}{5}$, from the definition of $I$-convergence we get $\left\{(m,n)\in\mathbb{N}\times\mathbb{N}:D_{mn}(A)[1-D_{mn}(A)]\geq\frac{1}{\epsilon^2}\right\}=M_2(say)\in\mathbb{F}(I).$ Taking $M_0=M_1\cap M_2$, we see that $M_0\in \mathbb{F}(I)$ and $D_{mn}(A)<\epsilon$, for all $(m,n)\in M_0.$ Therefore
\begin{center}
$\left\{(m,n):D_{mn}(A)\geq\epsilon\right\}\subset (\mathbb{N}\times\mathbb{N}\setminus M_0).$
\end{center}
Since $(\mathbb{N}\times\mathbb{N}\setminus M_0)\in I$ so $\left\{(m,n):D_{mn}(A)\geq\epsilon\right\}\in I$ and hence  $I-\lim\limits_{\stackrel{m,n\rightarrow\infty}~}D_{mn}(A)=0$.
Similarly if $D_{mn}(A)>\frac{4}{5}$ for all $(m,n)\in M_3\subset M$ for some $M_3\in\mathbb{F}(I).$ then we can show that $I-\lim\limits_{\stackrel{m,n\rightarrow\infty}~}D_{mn}(A)=1.$
\\ If neither of above cases happen then considering dictionary order on $\mathbb{N}\times\mathbb{N}$, we can find an increasing sequence
\begin{center}
$\left\{(m_1,n_1)<(m_2,n_2)<.....\right\}$
\end{center}
 from M such that
 \begin{center}
 $D_{m_in_i}<\frac{1}{5}$ when $i$ is odd integer.
 \end{center}
\begin{center}
$D_{m_in_i}>\frac{4}{5}$ when $i$ is even integer.
\end{center}
 Then clearly
 \begin{center}
 $\limsup\limits_{\stackrel{m,n\rightarrow\infty}~}D_{mn}(A)$-$\liminf\limits_{\stackrel{m,n\rightarrow\infty}~}D_{mn}(A)\geq \frac{3}{5}.$
 \end{center}
We again start the above process and see that
\begin{center}
$\left\{(m,n)\in\mathbb{N}\times\mathbb{N}:D_{mn}(A)[1-D_{mn}(A)]<\frac{1}{36}\right\}=M_4(say)\in \mathbb{F}(I).$
\end{center}
Which consequently implies as above that either $I-\lim\limits_{\stackrel{m,n\rightarrow\infty}~}D_{mn}(A)=1.$ or $I-\lim\limits_{\stackrel{m,n\rightarrow\infty}~}D_{mn}(A)=0.$ or if neither holds then
\begin{center}
$\limsup\limits_{\stackrel{m,n\rightarrow\infty}~}D_{mn}(A)$-$\liminf\limits_{\stackrel{m,n\rightarrow\infty}~}D_{mn}(A)\geq \frac{4}{6}.$
\end{center}
Proceeding in this way we observe that the process stops only when we get either  $I-\lim\limits_{\stackrel{m,n\rightarrow\infty}~}D_{mn}(A)=0.$ or  $I-\lim\limits_{\stackrel{m,n\rightarrow\infty}~}D_{mn}(A)=1.$ and if it does not stop at a finite step then we will have
\begin{center}
$\limsup\limits_{\stackrel{m,n\rightarrow\infty}~}D_{mn}(A)$-$\liminf\limits_{\stackrel{m,n\rightarrow\infty}~}D_{mn}(A)\geq \lim\limits_{\stackrel{k\rightarrow\infty}~}\frac{k-2}{k}.$
\end{center}
Which contradicts to our assumption that $0\leq r\leq 1$. This completes the proof of the theorem.
\end{proof}
\textbf{Remark 1} : For $A\subset \mathbb{N}\times\mathbb{N}$ if $I-\lim\limits_{\stackrel{m,n\rightarrow\infty}~}\frac{1}{mn}\left|\left\{(j,k)\in A:j\leq m, k\leq n\right\}\right|$ exists we can say that the $I$-asymptotic density of $A$ exists and denote it by $d_I (A)$. Therefore the above result can be re-phrased as: Let $x=\left\{x_{jk}\right\}_{j,k\in \mathbb N}$ be a double sequence of real numbers and $(\alpha,\beta)$ is an open interval such that $x_{jk}\notin(\alpha,\beta),$ for all $(j,k)\in \mathbb{N}\times\mathbb{N}$. We write $A=\left\{(j,k):x_{jk}\leq\alpha\right\}$ and further assume that the following property is satisfied
\begin{center}
$\limsup\limits_{\stackrel{m,n\rightarrow\infty}~}D_{mn}(A)$-$\liminf\limits_{\stackrel{m,n\rightarrow\infty}~}D_{mn}(A)<r.$
\end{center}
for some $0\leq r\leq 1.$ If $x$ is $I$-statistically pre-Cauchy then either $d_I(A)=0$ or $d_I(A)=1.$ It should be mentioned in this context that for $I = I_{fin}$, the above result holds without any additional assumption since it is easy to see thus we omitted. For our final result we assume that $I$ is such an ideal and $x$ is such that the above result holds without any
additional assumption i.e;
\\ $(**)$ If $x=\left\{x_{jk}\right\}_{j,k\in\mathbb N}$ is $I$-statistically pre-Cauchy double sequence of real numbers and $x_{jk}\notin(\alpha,\beta)$ for all $(j,k)\in\mathbb{N}\times\mathbb{N}$ then either $d_I(\left\{(j,k): x_{jk}\leq\alpha\right\})=0$ or $d_I(\left\{(j,k): x_{jk}\geq\beta\right\})=0$.
\\ Before we prove our final result, we introduce the following definition.
\begin{defn}
A real number $\xi$ is said to be an $I$-statistical cluster point of a double sequence $x=\left\{x_{jk}\right\}_{j,k\in \mathbb N}$ of real numbers if for any $\epsilon>0$
\begin{center}
$d_I(\left\{(j,k):\left|x_{jk}-\xi\right|<\epsilon\right\})\neq 0.$
\end{center}
\end{defn}
\begin{thm}
Let $x=\left\{x_{jk}\right\}$ be an $I$-statistically pre-Cauchy double sequence of real numbers. If the set of limit points of $x$ is no-where dense and $x$ has a $I$-statistical cluster point. Then $x$ is $I$-statistically convergent under the hypothesis $(**)$.
\end{thm}
\begin{proof}
Suppose $x$ has a $I$-statistical cluster point $\xi\in\mathbb R$ . So for any  $\epsilon>0$ we have $d_I(\left\{(j,k):\left|x_{jk}-\xi\right|<\epsilon\right\})\neq 0.$ Assume that $x$ is $I$-statistically pre-Cauchy satisfying the hypothesis $(**)$ but $x$ is not $I$-statistically convergent, so there is an $\epsilon_0>0$ such that  $d_I(\left\{(j,k):\left|x_{jk}-\xi\right|\geq\epsilon_0\right\})\neq 0.$  Without loss of generality we assume that $d_I(\left\{(j,k):x_{jk}\leq\xi-\epsilon_0\right\})\neq 0.$ We claim that every point of $(\xi-\epsilon_0,\xi)$ is a limit point of x. If not, then we can
find an interval $(\alpha,\beta)\subset(\xi-\epsilon_0,\xi)$ such that $x_{jk}\notin(\alpha,\beta)$ for all $(j,k)\in\mathbb{N}\times\mathbb{N}$. From above it
immediately follows that both $d_I(\left\{(j,k): x_{jk}\leq\alpha\right\})=0$ and $d_I(\left\{(j,k): x_{jk}\geq\beta\right\})=0$. But
this contradicts the hypothesis $(**)$. Hence every point of $(\xi-\epsilon_0,\xi)$ is a limit point of $x$ which contradicts that the set of limit points of $x$ is a no-where dense set. Hence $x$ is $I$-statistically convergent.
\end{proof}

 \textbf{Acknowledgment}
\\  The second author is thankful to UGC, India for Research fund.


\begin{thebibliography}{99}

\bibitem{1} C. Belen and M. Yildirim, {On generalized statistical convergence of double sequences via
ideals},\textit{ Ann. Univ. Ferrara} 58(1) (2012), 11-20.

\bibitem{2} J. Conor, J. Fridy, and J. Kline, Statistically pre-Cauchy Sequences, \textit{Analysis}, 14(1994) 311-317.

\bibitem{3} P. Das, P. Malik, On Extremal $I$-limit points of double sequences, \textit{Tatra Mt. Math. Publ.} 40 (2008), 91-102.

\bibitem{4} P. Das, , P. Kostyrko, W. Wilczy\'nski and P. Malik, $\mathcal{I}$ and $\mathcal{I}^{*}$
convergence of double sequences, \textit{Math.Slovaca},
58(5)(2008), 605-620.

\bibitem{5} P. Das, E. Savas, S.Kr. Ghosal, On generalizations of certain summability methods using ideals, \textit{Appl. Math. Lett.,}
24(2011) 1509-1514.

\bibitem{6} P. Das, E. Savas, On $I$-statistically pre-Cauchy sequences, \textit{Taiwanese J. Math}.,18(1)(2014) 115-126.

\bibitem{7}  K. Dutta, P. Malik, M. Maity, Statistical convergence of double sequences in probabilistic metric
spaces, \textit{Selcuk J. Math.}, vol-14, 1(2013) 57-70.

\bibitem{8} H. Fast, Sur la convergence statistique, \textit{Colloq. Math}., 2(1951) 241-244.

\bibitem{9} J.A. Fridy, On statistical convergence, \textit{Analysis}, 5(1985) 301-313.

\bibitem{10}  K. Menger, Statistical metric spaces, \textit{Proc. Nat. Acad of Sci., U.S.A}, 28(1942)
535-537.

\bibitem{11} M. Mursaleen, and O.H. Edely, H.H. Osama Statistical convergence of double sequences,\textit{ J. Math. Anal. Appl.},2003, 288, no.1, 223-231.

\bibitem{12}  A. Pringsheim, Zur theorie der Zweifach unendlichen Zahlenfolgen, \textit{Math. Ann., }53, no.(3)(1900) 289-321.

\bibitem{13}  T. S\'alat, P. Kostyrko, W. Wilczy\'nski, $I$-convergence, \textit{Real Anal. Exchange}, 26(2)(2000-2001) 669-685.

\bibitem{14} E. Savas, P. Das, A generalized statistical convergence via ideals, \textit{Appl. Math. Lett.}, 24(2011)
826-830.

\bibitem{15} I. J. Schoenberg, The integrability of certain
functions and related summability methods, \textit{ Amer. Math.
Monthly}, 66(1959), 361-375.

\bibitem{16}  B. Schweizer,  and  A. Sklar, Statistical metric spaces, \textit{Pacific J. Math.}, 10(1960) 313-334.

\bibitem{17} B. Schweizer,  and  A. Sklar, and  E. Thorp, The metrization of statistical metric spaces, \textit{Pacific J. Math}. 10(1960) 673-675.

\bibitem{18} B. Schweizer,  and  A. Sklar, Statistical metric spaces arising from sets of random variables in Euclidean n-space, \textit{Theory of probability and its Applications,} 7(1962) 447-456.

\bibitem{19} B. Schweizer,  and  A. Sklar, tringle inequalities in a class of statistical metric spaces, \textit{J. London Math. Soc.}, 38(1963) 401-406.

\bibitem{20} B. Schweizer,  and  A. Sklar, Probabilistic metric spaces, \textit{(North-Holland Publishing Co., New York).}, 1983.

\bibitem{21} C. \c{S}en\c{c}imen, S. Pehlivan, strong ideal convergence in Probabilistic metric
spaces, \textit{Proc. Indian Acad. Sci. (Math. Sci)}, 119(3)(2009) 401-410.

\bibitem{22} H. Steinhus, Sur la convergence ordinatre et la convergence asymptotique, \textit{Colloq. Math}., 2(1951) 73-74.

\bibitem{23}  R. M. Tardiff, Topologies for Probabilistic Metric spaces, \textit{Pacific J. Math.}, 65(1976) 233-251.

\bibitem {24} E. Thorp, Generalized topologies for ststistical
metric spaces, \textit{Fundamenta Mathematicae}, 51(1962), 9-21.

\bibitem{25} A. Zygmund, Trigonometric Series, Second ed., \textit{Cambridge Univ. Press}, 1979.


\end{thebibliography}
\end{document}